\documentclass[a4paper]{article}

\usepackage{amsmath, amsthm, amssymb, mathrsfs}
\usepackage{graphicx, tikz}
\usepackage{hyperref}

\newtheorem{definition}{Definition}[section]
\newtheorem{problem}{Problem}[section]
\newtheorem{theorem}{Theorem}[section]
\newtheorem{remark}{Remark}[section]

\begin{document}

\title{Uniqueness in phaseless inverse scattering problems with
superposition of incident point sources}

\author{
Fenglin Sun\thanks{School of Mathematics, Jilin University, Changchun, P. R. China {\it sunfl18@mails.jlu.edu.cn}}, 
Deyue Zhang\thanks{School of Mathematics, Jilin University, Changchun, P. R. China {\it dyzhang@jlu.edu.cn} }\ \ and 
Yukun Guo\thanks{Department of Mathematics, Harbin Institute of Technology, Harbin, P. R. China. {\it ykguo@hit.edu.cn} (Corresponding author)}
}

\date{}

\maketitle

\begin{abstract}
This paper is concerned with the uniqueness in inverse acoustic scattering problems with the modulus of the far-field patterns co-produced by the obstacle (resp. medium) and the point sources. Based on the superposition of point sources as the incident waves, we overcome the difficulty of translation invariance induced by a single incident plane wave, and rigorously prove that
the location and shape of the obstacle as well as its boundary condition or the refractive index can be uniquely determined by the modulus of far-field patterns. This work is different from our previous work on phaseless inverse scattering problems [2018 {\it Inverse Problems} 34, 085002], in which the reference ball technique and the superposition of incident waves were used, and the phaseless far-field data generated only by the the scatterer were considered. In this paper, the phaseless far-field data co-produced by the scatterer and the point sources are used, thus the configuration is practically more feasible. Moreover, since the reference ball is not needed, the justification of uniqueness is much more clear and concise.
\end{abstract}



\section{Introduction}

The inverse scattering problems with the measurment of the far-field pattern is of significant importance in diverse areas of sciences and technology such as invasive testing, target identification, radar sensing, sonar detection and biomedical imaging (see, e.g., \cite{Colton}). However, in many realistic situations, only intensity/modulus information of the far-field pattern might be measured, which leads to the study of inverse scattering problems with phaseless or intensity-only far-field data. Owing to the lack of phase information, the mathematical justifications and numerical reconstructions of phaseless inverse scattering problems are more challenging than the phased case \cite{ACZ16}. In this study, we will deal with the uniqueness issue concerning the phaseless inverse acoustic scattering problems with incident point sources.

A well-known difficulty in phaseless inverse acoustic scattering with one plane wave as the incident field is the non-uniqueness. For example,  the location of the obstacle (resp. medium)
cannot be uniquely determined by the phaseless far-field data, which is due to the so-called translation invariance property, i.e., the modulus of the far field pattern is invariant under
translations \cite{Kli14, Kli17, KR17, Kress, Liu}. And even more notoriously, the invariance property of the phaseless far-field pattern cannot be remedied by using finitely many incident waves with different
wave numbers or different incident directions (see, e.g., \cite{Kress}). Consequently, some numerical algorithms for the shape reconstruction have been developed, such as the Newton method \cite{Kress}, the nonlinear integral equation method \cite{ GDM18, Ivanyshyn1, Ivanyshyn2, Ivanyshyn3}, the fundamental solution method \cite{KarageorghisAPNUM},
the hybrid method \cite{Lee2016}, the reverse time migration method \cite{CH17} and the direct imaging method \cite{ZZ18}. We also refer to \cite{BLL2013, Bao2016, KR16, Li1, LLW17, Shin} for the reconstruction of the shape of
a polyhedral obstacle, a convex sound-soft scatterer, a periodic grating profile and multi-scale sound-soft rough surfaces from phaseless far-field or near-field data. For a small sound-soft ball centered at the origin, the uniqueness on radius of the ball from a single phaseless far-field datum was established in \cite{LZ09}.

To break the translation invariance, the authors in \cite{ZhangBo20171} used the method of the superposition of two plane waves with different incident directions, and proposed a recursive Newton-type iteration algorithm. This method was also developed in \cite{ZhangBo20172} to recover locally rough surfaces with phaseless far-field data. By the method of superposition of two incident plane waves, uniqueness results were proved in \cite{XZZ18a} under a priori assumption.

The idea of resorting to the reference ball technique (see, e.g., \cite{Colton1, Colton2, Li}) in phaseless inverse scattering problems was proposed by Zhang and Guo in \cite{ZhangDeyue20181}. For a general setting of the underlying scatterer, the uniqueness results were established in \cite{ZhangDeyue20181} by utilizing the reference ball technique in conjunction with the superposition of a fixed plane wave and some point sources as the incident wave. With the aid of the reference ball technique, the a priori assumptions in \cite{XZZ18a} can be removed as well, see \cite{XZZ18b} for the details. In addition to the theoretical interests in phaseless inverse scattering, the strategy of adding reference objects or sources to the scattering system has also received an increasing attention in devising effective numerical inversion schemes. For instance, the reference ball based iterative methods have been developed in \cite{DZG19} to reconstruct both the location and shape of a scattering acoustic obstacle. A reference point scatterer based direct sampling method can be found in \cite{JLZ18a}. Similarly, the idea of the reference source was investigated in \cite{ZhangDeyue20182,JLZ18b} for solving some phaseless inverse source problems. We also refer to \cite{DLL18, JL18} for some recent studies on utilizing artificial reference objects in phaseless inverse elastic scattering problems.

It should be remarked that the measured phaseless far-field data in \cite{ZhangDeyue20181} are only generated by the the obstacle (resp. medium). However, in realistic applications, the far-field patterns are co-produced by the obstacle (resp. medium) and the point sources, and the corresponding phaseless data cannot be decoupled and measured separated. In this paper,
based on the superposition of point sources as the incident field and the measurements of the phaseless far-field patterns co-produced by the obstacle (resp. medium)
and the point sources, we rigorously prove that the location and shape of the obstacle as well as its boundary condition or the refractive index can be uniquely determined by
the modulus of far-field patterns. We would like to emphasize that the co-produced phaseless far-field patterns provide more information and overcome the indispensability of incorporating a reference ball in the previous works on establishing the uniqueness. 

In comparison with the previous reference ball technique in \cite{ZhangDeyue20181}, the treatment in this study exhibits the superiorities in the following aspects: First, this configuration enables us to provide a unified framework for the justification of uniqueness in both the cases of penetrable and impenetrable scatterers. Second, the proofs are significantly simplified since the essential requirment of the reference ball turns out to be completely unnecessary. Finally, from the practical point of view, to access the modulus of the far-field patterns concurrently produced by the underlying scatterer and the point sources should be relatively easier than capturing the modulus of far-field due solely to the scatterer. 

The rest of this paper is arranged as follows. In the next section, we present an introduction to the model problem. Section \ref{sec:obstacle} is devoted to the uniqueness results on phaseless inverse scattering problem.

\section{Problem setting}\label{sec:problem_setup}

We begin this section with the acoustic scattering problems for an incident plane wave. Assume $D \subset\mathbb{R}^3$ is an open and simply connected domain with
$C^2$ boundary $\partial D$.  Denote by $\nu$ be the unit outward normal to $\partial D$ and by $\mathbb{S}^2:=\{x\in\mathbb{R}^3: |x|=1\}$ the unit sphere in $\mathbb{R}^3$.
Let $u^i(x,d)=\mathrm{e}^{\mathrm{i} k x\cdot d}$ be a given incident plane wave, where $d\in\mathbb{S}^2$ and $k>0$ are the incident direction and wave number, respectively.
Then, the obstacle scattering problem can be formulated as: to find the total field $u=u^i+u^s$ which satisfies the following boundary value problem (see \cite{Colton}):
\begin{eqnarray}
\Delta u+ k^2 u= 0\quad \mathrm{in}\ \mathbb{R}^3\backslash\overline{D},\label{eq:Helmholtz} \\
\mathscr{B}u= 0 \quad \mathrm{on}\ \partial D, \label{eq:boundary_condition} \\
\lim\limits_{r=|x|\rightarrow\infty} r\bigg(\displaystyle\frac{\partial u^s}{\partial r} -\mathrm{i} ku^s\bigg)=0, \label{eq:Sommerfeld}
\end{eqnarray}
where $u^s$ denotes the scattered field and \eqref{eq:Sommerfeld} is the Sommerfeld radiation condition. Here $\mathscr{B}$ in \eqref{eq:boundary_condition} is the boundary operator
defined by
\begin{eqnarray}\label{BC}
\mathscr{B}u:=
\left\{
\begin{array}{lr}
u & \mathrm{on}\ \partial D_D, \\
\displaystyle\frac{\partial u}{\partial \nu}+ \lambda u, & \mathrm{on}\ \partial D_I,
\end{array}
\right.
\end{eqnarray}
where $\partial D_D\cup\partial D_I=\partial D, \partial D_D\cap\partial D_I=\emptyset, \lambda\in C(\partial D_I)$ and $\mathrm{Im}\lambda \geq 0$.
This mixed boundary condition \eqref{BC} covers the Dirichlet/sound-soft boundary condition ($\partial D_I=\emptyset$), the Neumann/sound-hard boundary condition ($\partial D_D=\emptyset$ and $\lambda=0$), and the impedance boundary condition ($\partial D_D=\emptyset$ and $\lambda\neq 0$).

The medium scattering problem is to find the total field $u=u^i+u^s$ that fulfills
\begin{eqnarray}
    \Delta u+ k^2n(x) u=  0 \quad \text{in}\ \mathbb{R}^3, \label{eq:Helmholtz_D}\\
    \lim\limits_{r=|x|\rightarrow\infty} r\bigg(\displaystyle\frac{\partial u^s}{\partial r}-\mathrm{i} ku^s\bigg)=0, \label{eq:Sommerfeld2}
\end{eqnarray}
where the refractive index $n(x)$ of the inhomogeneous medium is piecewise continuous such that $\mathrm{Re}(n)>0$, $\mathrm{Im}(n)\geq0$ and $1-n(x)$ is supported in $D$.

The direct scattering problems \eqref{eq:Helmholtz}--\eqref{eq:Sommerfeld} and \eqref{eq:Helmholtz_D}--\eqref{eq:Sommerfeld2} admit a unique solution
(see, e.g., \cite{Cakoni, Colton, McLean}), respectively, and the scattered wave $u^s$ has the following asymptotic behavior
$$
u^s(x,d)=\frac{\mathrm{e}^{\mathrm{i} k|x|}}{|x|}\left\{ u^{\infty}(\hat{x},d)+\mathcal{O}\left(\frac{1}{|x|}\right) \right\}, \quad |x|\rightarrow\infty
$$
uniformly in all observation directions $\hat{x}=x/|x|\in\mathbb{S}^2$. The analytic function $u^{\infty}(\hat{x},d)$ defined on the unit sphere $\mathbb{S}^2$ is called the far field
pattern or scattering amplitude (see \cite{Colton}). 

Now, the phaseless inverse scattering problem is to determine the location and shape $\partial D$ as well as the boundary condition $\mathscr{B}$
for the obstacle or the refractive index $n$ for the medium inclusion by the phaseless far field data $|u^{\infty}(\hat{x}, d)|$ for $\hat{x}, d\in \mathbb{S}^2$
and a fixed $k>0$. As we mentioned, there is no uniqueness on the phaseless inverse scattering problem due to the translation invariance, and
this ambiguity cannot be remedied by using finitely many incident waves with different wave numbers or different incident directions.

In the following, we will try to break the translation invariance by utilizing the superposition of incident point sources. To this end, we first introduce the following definition of admissible surfaces.
\begin{definition}[Admissible surface]
An open surface $\Gamma$ is called admissible if there exists a simply-connected domain $\Omega$ such that

\noindent (i) $\Omega\subset\mathbb{R}^3\backslash\overline{D}$ and $\partial \Omega$ is analytic homeomorphic to $\mathbb{S}^2$;

\noindent (ii) $k^2$ is not a Dirichlet eigenvalue of $-\Delta$ in $\Omega$;

\noindent (iii) $\Gamma\subset\partial\Omega$ is a two-dimensional analytic manifold with nonvanishing measure. 
\end{definition}

\begin{remark}
	We would like to point out that this requirement for the admissibility of $\Gamma$ is quite mild and thus can be easily fullfilled. For instance, $\Omega$ can be chosen as a ball whose radius is less than $\pi/k$ and $\Gamma$ is chosen as an arbitrary corresponding semisphere.
\end{remark}

For a generic point $z\in\mathbb{R}^3\backslash\overline{D}$, the incident field due to the point source located at $z$ is given by
\begin{equation*}
\Phi (x, z):=\frac{\mathrm{e}^{\mathrm{i} k|x-z|}}{4\pi |x-z|}, \quad x\in\mathbb{R}^3\backslash(\overline{D}\cup\{z\}),
\end{equation*}
which is also known as the fundamental solution to the Helmholtz equation. Denote by $v^{\infty}_D(\hat{x},z)$ the far-field patterns generated by $D$ corresponding to the incident field $\Phi(x, z)$. Define
$$
v^{\infty}(\hat{x},z):=v^{\infty}_D(\hat{x},z)+\Phi^\infty(\hat{x}, z),\quad \hat{x}\in\mathbb{S}^2,
$$
where $\Phi^\infty(\hat{x},z):=\mathrm{e}^{-\mathrm{i} k \hat{x}\cdot z}/(4\pi)$ is the the far-field pattern of $\Phi(x, z)$.
\begin{remark}
	As far as we know, the phaseless data $|v^{\infty}_D(\hat{x},z)+\Phi^\infty(\hat{x}, z)|$ is usually obtainable in realistic scenarios, whereas the phaseless data $|v^{\infty}_D(\hat{x},z)|$ cannot be measured directly. Hence, in our view, it would make more sense to use intensity-only data $|v^{\infty}(\hat{x},z)|=|v^{\infty}_D(\hat{x},z)+\Phi^\infty(\hat{x}, z)|$.
\end{remark}
For two generic and distinct source points $z_1, z_2\in\mathbb{R}^3\backslash\overline{D}$, we denote by
\begin{equation}\label{incident}
 v^i(x; z_1,z_2):=\Phi(x, z_1)+\Phi(x, z_2),\quad x\in\mathbb{R}^3\backslash(\overline{D}\cup\{z_1\}\cup\{z_2\}),
\end{equation}
the superposition of these point sources. Then, by the linearity of direct scattering problem, the far-field pattern co-produced by $D$ and the incident wave $v^i(x; z_1,z_2)$ is given by
$$
v^{\infty}(\hat{x};z_1,z_2):=v^{\infty}(\hat{x},z_1)+v^{\infty}(\hat{x},z_2), \quad \hat{x}\in \mathbb{S}^2.
$$

With these preparations,  we formulate the phaseless inverse scattering problems as the following.

\begin{problem}[Phaseless inverse scattering]\label{prob:obstacle}
Assume that $\Gamma$ is an admissible surface. Let $D$ be the impenetrable obstacle with boundary condition $\mathscr{B}$ or the inhomogeneous medium with refractive index $n$. Given the phaseless far field data
 \begin{equation*}
 \begin{array}{ll}
 & \{|v^{\infty}(\hat{x},z_0)|: \hat{x}\in \mathbb{S}^2\}, \\
 & \{|v^{\infty}(\hat{x},z)|:  \hat{x}\in \mathbb{S}^2,\ z\in \Gamma\}, \\
 & \{|v^{\infty}(\hat{x},z_0)+v^{\infty}(\hat{x},z)|: \hat{x}\in \mathbb{S}^2,\ z\in \Gamma\}
 \end{array}
 \end{equation*}
 for a fixed wavenumber $k>0$ and a fixed  $z_0\in\mathbb{R}^3\backslash(\overline{D}\cup\Gamma)$, determine the location and shape $\partial D$ as well as the boundary condition
 $\mathscr{B}$ for the obstacle or the refractive index $n$ for the medium inclusion.
\end{problem}

We refer to Figure \ref{fig:illustration} for an illustration of the geometry setting of Problem \ref{prob:obstacle}. The uniqueness of this problem will be analyzed in the next section.

\begin{figure}
    \centering
    \newdimen\R 
    \R=0.5cm
    \begin{tikzpicture}[thick]
    \pgfmathsetseed{3}
    \draw plot [smooth cycle, samples=7, domain={1:8},  xshift=0.7cm, yshift=-0.1cm] (\x*360/8+5*rnd:0.5cm+1cm*rnd) [fill=lightgray] node at (0.6,-0.2) {$D$};
    \draw [blue, dashed] (1.6, 3.4) circle (0.6cm); 
    \draw node at (1.6,3.5) {$\Omega$};
    \draw (1.05, 3.15) arc(205:345:0.6cm) [very thick, blue];
    \draw node at (2.3,2.9) {$\Gamma$};
    
    \fill [red] (-1.5,1.5) circle (2pt);  
    \draw (-1.5, 1.5) node [above] {$z_0$};
    \draw [->] (-1.4, 1.4)--(-0.8, 0.9);
    \draw (-1.6, 1.2) arc(280:360:0.5cm);
    \draw (-1.6, 1.0) arc(280:360:0.7cm);
    
    \fill [red] (1.6,2.8) circle (2pt);  
    \draw node at (1.7,3.0) {$z$};
    \draw [->] (1.58, 2.7)--(1.5, 2);
    \draw (1.2, 2.6) arc(225:315:0.5cm);
    \draw (1.1, 2.5) arc(220:320:0.6cm);
    
    \draw (-0.2, 2) node [above] {$v^i(\cdot; z_0, z)$}; 
    \draw [->] (2.5,1.1)--(3.4, 1.7);
    \draw (3.0, 1) arc(0:70:0.6cm);
    \draw (3.2, 1) arc(0:70:0.8cm);  
      
    \draw [dashed] (1, 1) circle (3.2cm) node at (5, 3.8) {$|v^\infty(\cdot; z_0, z)|, |x|\to \infty$}; 
    \end{tikzpicture}
    \caption{An illustration of the phaseless inverse scattering problem.} \label{fig:illustration}
\end{figure}
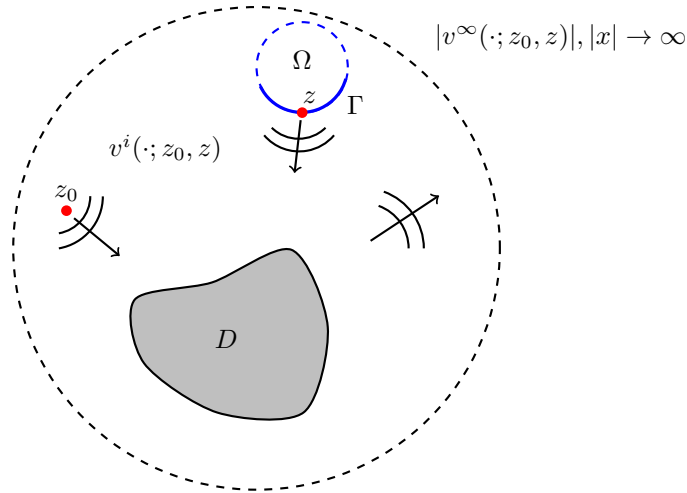

\section{Uniqueness for the phaseless inverse scattering}\label{sec:obstacle}

Now we present the uniqueness results on phaseless inverse scattering. The following theorem shows that Problem \ref{prob:obstacle} admits a unique solution, namely, the geometrical and physical information of the scatterer boundary or the refractive index for the medium can be simultaneously and uniquely determined from the modulus of far-field patterns.

\begin{theorem}\label{Thm1}
   For two scatterers $D_1$ and $D_2$, suppose that the corresponding far-field patterns satisfy that
   \begin{eqnarray}
   &&|v^{\infty}_1(\hat{x},z_0)|= |v^{\infty}_2(\hat{x},z_0)|, \quad \forall \hat{x}  \in \mathbb{S}^2, \label{obstacle_condition1} \\
   &&\ |v^{\infty}_1(\hat{x},z)|= |v^{\infty}_2(\hat{x},z)|, \quad \forall (\hat{x}, z) \in \mathbb{S}^2\times\Gamma \label{obstacle_condition2}
   \end{eqnarray}
   and
   \begin{eqnarray}\label{obstacle_condition3}
   |v^{\infty}_1(\hat{x},z_0)+v^{\infty}_1(\hat{x},z)|=|v^{\infty}_2(\hat{x},z_0)+v^{\infty}_2(\hat{x},z)|,\quad \forall (\hat{x}, z) \in \mathbb{S}^2\times\Gamma
   \end{eqnarray}
   for an admissible surface $\Gamma$ and an arbitrarily fixed $z_0\in\mathbb{R}^3\backslash(\overline{D}\cup\Gamma)$. Then we have
   
\noindent (i) If $D_1$ and $D_2 $ are two impenetrable obstacles with boundary conditions $\mathscr{B}_1$ and $\mathscr{B}_2$ respectively, then $D_1=D_2$ and $\mathscr{B}_1=\mathscr{B}_2$.

\noindent (ii) If $D_1$ and $D_2 $ are two medium inclusions with refractive indices $n_1$ and $n_2$ respectively, then $n_1=n_2$.
\end{theorem}
\begin{proof}
    From \eqref{obstacle_condition1}, \eqref{obstacle_condition2} and \eqref{obstacle_condition3}, we have for all $\hat{x}\in\mathbb{S}^2, z\in\Gamma$
    \begin{equation}\label{Thm1equality1}
    \mathrm{Re}\left\{v^{\infty}_1(\hat{x},z_0) \overline{v^{\infty}_1(\hat{x},z)}\right\}
    =\mathrm{Re}\left\{v^{\infty}_2(\hat{x},z_0) \overline{v^{\infty}_2(\hat{x},z)}\right\},
    \end{equation}
    where the overline denotes the complex conjugate. According to \eqref{obstacle_condition1} and \eqref{obstacle_condition2}, we denote
    \begin{equation*}
        v^{\infty}_j(\hat{x},z_0)=r(\hat{x},z_0) \mathrm{e}^{\mathrm{i} \alpha_j(\hat{x},z_0)},\quad
        v^{\infty}_j(\hat{x},z)=s(\hat{x},z) \mathrm{e}^{\mathrm{i} \beta_j(\hat{x},z)},\quad j=1,2,
    \end{equation*}
    where $r(\hat{x},z_0)=|v^{\infty}_j(\hat{x},z_0)|, s(\hat{x},z)=|v^{\infty}_j(\hat{x},z)|$,
    $\alpha_j(\hat{x},z_0)$ and $\beta_j(\hat{x},z)$, are real-valued functions,  $j=1,2$.

    From $s(\hat{x},z)\not\equiv 0$ for $\hat{x} \in \mathbb{S}^2, z \in\Gamma$, and its continuity, one knows that there exists an open set
    $\tilde{S}\subset \mathbb{S}^2$ such that $s(\hat{x},z)\neq 0$, $\forall (\hat{x},z)\in \tilde{S}\times\Gamma$.
    Since $v^{\infty}_j(\hat{x},z_0)$ are analytic with respect to $\hat{x} \in \mathbb{S}^2$,  $v^{\infty}_j(\hat{x},z_0)\not\equiv 0$ for $\hat{x} \in \mathbb{S}^2$, and $z_0$ is fixed,
    we have $r(\hat{x},z_0)\not\equiv 0$ on $\tilde{S}$. Again, the continuity leads to $r(\hat{x},z_0)\neq 0$ on an open set $S \subset \tilde{S}$. Therefore, we have $r(\hat{x},z_0)\neq 0$,
    $s(\hat{x},z)\neq 0,\ \forall (\hat{x}, z) \in S\times\Gamma$. In addition, taking \eqref{Thm1equality1} into account, we derive that
    \begin{equation*}
        \cos[\alpha_1(\hat{x},z_0)-\beta_{1}(\hat{x},z)]=\cos[\alpha_2(\hat{x},z_0)-\beta_{2}(\hat{x},z)], \quad \forall (\hat{x}, z) \in S\times\Gamma.
    \end{equation*}
    Hence, either
    \begin{eqnarray}\label{Thm1equality2}
    \alpha_1(\hat{x},z_0)-\alpha_2(\hat{x},z_0)=\beta_{1}(\hat{x},z)-\beta_{2}(\hat{x},z)+ 2m\pi, \quad \forall (\hat{x}, z) \in S\times\Gamma
    \end{eqnarray}
    or
    \begin{eqnarray}\label{Thm1equality3}
    \alpha_1(\hat{x},z_0)+\alpha_2(\hat{x},z_0)=\beta_1(\hat{x},z)+\beta_{2}(\hat{x},z)+ 2m\pi, \quad \forall (\hat{x}, z) \in S\times\Gamma
    \end{eqnarray}
    holds with some $m \in \mathbb{Z}$.

    First, we shall consider the case \eqref{Thm1equality2}. Since $z_0$ is fixed, let us define $\gamma(\hat{x}):=\alpha_1(\hat{x},z_0)-\alpha_2(\hat{x},z_0)- 2m\pi$
    for $ \hat{x} \in S$, and then, we deduce for all $ (\hat{x}, z) \in S\times\Gamma$
    \begin{equation*}
        v^{\infty}_1(\hat{x},z)=s(\hat{x},z)\mathrm{e}^{\mathrm{i} \beta_{1}(\hat{x},z)}
        =s(\hat{x},z)\mathrm{e}^{\mathrm{i} \beta_{2}(\hat{x},z)+\mathrm{i} \gamma(\hat{x})}=v^{\infty}_2(\hat{x},z)\mathrm{e}^{\mathrm{i} \gamma(\hat{x})}.
    \end{equation*}
    From the mixed reciprocity relation \cite[Theorem 3.16]{Colton} for the obstacle or \cite[Theorem 2.2.4]{Pot01} for the inhomogeneous medium, we have
    \begin{equation*}
        4\pi v^{\infty}_{D_j}(\hat{x},z)= u^s_j(z,-\hat{x}), \quad j=1,2,\quad \forall (\hat{x}, z) \in S\times\Gamma.
    \end{equation*}
where $u^s_j(z,-\hat{x})$ denotes the scattered field induced by $D_j$ and the incident plane wave with impinging direction $-\hat{x}$. Thus,
    \begin{equation*}
        u^{s}_1(z,-\hat{x})+u^i(z,-\hat{x})
        =\mathrm{e}^{\mathrm{i} \gamma(\hat{x})}\left[ u^{s}_2(z,-\hat{x})+u^i(z,-\hat{x})\right], \quad \forall (\hat{x}, z) \in S\times\Gamma.
    \end{equation*}
Let $u_j=u^s_j+u^i(j=1,2)$ be the total field. Then, for every $-d\in S$, it holds that $u_1(z,d)=\mathrm{e}^{\mathrm{i} \gamma(-d)} u_2(z,d)$ for $z\in\Gamma$. By using the analyticity of
$u_j(z,d)$($j=1,2$), we have $u_1(z,d)=\mathrm{e}^{\mathrm{i} \gamma(-d)} u_2(z,d)$ for $z\in \partial\Omega$.
Let $w(x,d)=u_1(x,d)-\mathrm{e}^{\mathrm{i} \gamma(-d)}  u_2(x,d)$, then
\begin{equation*}
\left\{
\begin{array}{lr}
\Delta w+ k^2 w=0 &\mathrm{in}\ \Omega, \\
w=0  &\mathrm{on}\ \partial \Omega.
\end{array}
\right.
\end{equation*}
By the assumption of $\Omega$ that $k^2$ is not a Dirichlet eigenvalue of $-\Delta$ in $\Omega$,  we find $w=0$ in $\Omega$. Now, the analyticity of
$u_j(x,d)(j=1,2)$ with respect to $x$ yields
    \begin{equation*}
        u_1(x,d)=\mathrm{e}^{\mathrm{i} \gamma(-d)}  u_2(x,d),\quad \forall x \in \mathbb{R}^3\backslash (\overline{D}_1\cup \overline{D}_2).
    \end{equation*}
i.e.,
\begin{equation}\label{eq:relation}
        u^s_1(x,d)+u^i(x,d)=\mathrm{e}^{\mathrm{i} \gamma(-d)} \left[ u^s_2(x,d)+u^i(x,d)\right],\quad \forall x \in \mathbb{R}^3\backslash (\overline{D}_1\cup \overline{D}_2).
\end{equation}

By the Green's formula \cite[Theorem 2.5]{Colton}, one can readily deduce that
$$
\lim_{|x|\to \infty} u_j^s(|x|\hat x,d)=0,\quad j=1,2,
$$
holds uniformly in all directions $\hat x=x/|x|$. Hence, by letting $|x|\rightarrow \infty$ in \eqref{eq:relation}, we obtain 
$$
0=\lim_{|x|\to \infty}[u^s_1(x, d)-\mathrm{e}^{\mathrm{i} \gamma(-d)}u^s_2(x, d)]=\lim_{|x|\to \infty}(\mathrm{e}^{\mathrm{i} \gamma(-d)}-1)\mathrm{e}^{\mathrm{i}kx\cdot d}.
$$
uniformly in all directions $\hat x=x/|x|$. Therefore, $\mathrm{e}^{\mathrm{i} \gamma(-d)} \equiv1$, and
\begin{equation*}
    u^{\infty}_1(\hat{x},d)= u^{\infty}_2(\hat{x},d), \quad \forall (\hat{x},-d)\in \mathbb{S}^2\times S.
\end{equation*}
    Further, the reciprocity relation and the analyticity of $u^{\infty}_j(-d,-\hat{x})$($j=1,2$) with respect to $d \in \mathbb{S}^2$ implies that the far field patterns coincide, i.e.,
    \begin{eqnarray}\label{coincide}
    u^{\infty}_1(\hat{x},d)= u^{\infty}_2(\hat{x},d),
    \quad \forall \hat{x}, d \in \mathbb{S}^2.
    \end{eqnarray}

Next we are going to show that the case \eqref{Thm1equality3} does not hold. Suppose that \eqref{Thm1equality3} is true, then following a similar argument, we see that there exists $\eta(-d)$ such that $u_1(x,d)=\mathrm{e}^{\mathrm{i} \eta(-d)} \overline{u_2(x,d)}$ for  $-d  \in S$ and $x \in \mathbb{R}^3\backslash (\overline{D}_1\cup\overline{D}_2)$. Again, taking $|x|\rightarrow \infty$, we have
\begin{equation*}
         \lim\limits_{|x|\rightarrow \infty}\mathrm{e}^{2\mathrm{i} k x\cdot d}=\mathrm{e}^{\mathrm{i} \eta(-d)},
\end{equation*}
which is a contradiction. Therefore, the case \eqref{Thm1equality3} does not hold.

Having verified \eqref{coincide}, we shall complete our proof as the consequences of two existing uniqueness results. For the inverse obstacle scattering, by Theorem 5.6 in \cite{Colton}, we have $D_1=D_2$ and $\mathscr{B}_1=\mathscr{B}_2$, and for
inverse medium scattering, Theorem 10.5 in \cite{Colton} leads to  $n_1=n_2$.
\end{proof}

\begin{remark}
    We would like to point out that an analogous uniqueness result in two dimensions remains valid after appropriate modifications of the fundamental solution, the radiation condition and the admissible surface. So we omit the 2D details.
\end{remark}

\begin{remark}
   We would like to remark that a similar result on uniqueness can be also obtained by using the superposition of a fixed plane wave and some point sources as the incident fields.
\end{remark}

\section*{Acknowledgements}

F. Sun and D. Zhang were supported by NSFC grant 11671170 and Y. Guo was supported by NSFC grants 11601107, 41474102 and 11671111.


\end{document}